\documentclass[12pt]{amsart}

\usepackage{amssymb,url,mathrsfs,euscript}

\usepackage[british]{babel}
\usepackage[T1]{fontenc} 

\newtheorem{teo}{Theorem}[section]
\newtheorem{lema}{Lemma}[section]
\newtheorem{pro}{Proposition}[section]

\newtheorem{rmk}{Remark}[section]\newtheorem*{rmk*}{Remark}

\newtheorem{coro}{Corollary}[section]

\theoremstyle{remark}     
\newtheorem{ex}[rmk]{Example}\newtheorem*{ex*}{Example}
\newtheorem*{exs*}{Examples}

\newtheorem*{acknowledgements}{Acknowledgements}

\def\sideremark#1{\ifvmode\leavevmode\fi\vadjust{\vbox to0pt{\vss
\hbox to 0pt{\hskip\hsize\hskip1em%
\vbox{\hsize2cm\tiny\raggedright\pretolerance10000%
\noindent #1\hfill}\hss}\vbox to8pt{\vfil}\vss}}}%

\swapnumbers              
\theoremstyle{plain}      
\newtheorem{lemma}[rmk]{Lemma}

\newtheorem{theorem}[rmk]{Theorem}

\theoremstyle{definition} 

\newcommand{\bt}{\begin{theorem}}\newcommand{\et}{\end{theorem}}
\newcommand{\bl}{\begin{lemma}}\newcommand{\el}{\end{lemma}}
\newcommand{\bp}{\begin{proof}}\newcommand{\ep}{\end{proof}}

\newcommand{\be}{\begin{equation}}\newcommand{\ee}{\end{equation}}
\newcommand{\bdm}{\begin{displaymath}}
\newcommand{\edm}{\end{displaymath}}

\numberwithin{equation}{section}

\def \tnabla{\widetilde{\nabla}}
 
\def \tiR{\widetilde{R}}

\def \z{\zeta}

\newcommand{\iso}{\cong}

\newcommand{\lan}{\langle}
\newcommand{\ran}{\rangle}
\newcommand{\hook}{\lrcorner\,}

\newcommand{\D}{\curly{D}}
\renewcommand{\o}{\omega}

\renewcommand{\leq}{\leqslant}
\newcommand{\R}{\mathbb{R}}
\newcommand{\C}{\mathbb{C}}

\newcommand{\la}{\lambda}\newcommand{\La}{\Lambda}
\newcommand{\w}{\wedge}
\renewcommand{\ln}{\textsl{ln}\,}

\newcommand{\lie}[1]{\mathfrak{#1}}
\newcommand{\Lie}[1]{\textsl{#1}}

\DeclareMathOperator{\SO}{\Lie{SO}}
\DeclareMathOperator{\SL}{\Lie{SL}}

\DeclareMathOperator{\U}{\Lie{U}}

\DeclareMathOperator{\so}{\lie{so}}

\renewcommand{\Re}{\textsl{Re}}
\renewcommand{\Im}{\textsl{Im}}
\DeclareMathOperator{\Ric}{\textsl{Ric}\,}
\DeclareMathOperator{\s}{\textsl{s}}

\renewcommand{\span}{\textsl{span}}
\newcommand{\grad}{\textsl{grad\,}}

\newcommand{\curly}{\mathscr}

\newcommand{\q}{\quad}\newcommand{\qq}{\qquad}

\newcommand{\bb}{\mathbb}
\newcommand{\ba}{\begin{array}}\newcommand{\ea}{\end{array}}

\renewcommand{\&}{{\footnotesize \&}}

\begin{document}

\title[]{Systems of symplectic forms \\ on four-manifolds}
\thanks{Partially supported by {\sc Gnsaga} of {\sc In}d{\sc am}, {\sc Prin} \oldstylenums{2007} of {\sc Miur} (Italy), and the Royal Society of New 
Zealand, Marsden grant no. 06-UOA-029}
\date{\today}
\author[S.G.Chiossi]{Simon G. Chiossi} 
\address[SGC]{Dipartimento di matematica, Politecnico di Torino, c.so 
Duca degli Abruzzi 24, 10129 Torino, Italy}
\email{simon.chiossi@polito.it}

\author[P.-A.Nagy]{Paul-Andi Nagy} 
\address[PAN]{Institut f\"ur Mathematik und Informatik, Ernst-Moritz-Arndt Universit\"at Greifs\-wald, Walther Rathenau Str. 47,
17487 Greifswald, Germany}
\email{nagyp@uni-greifswald.de}
\frenchspacing

\begin{abstract}
We study almost Hermitian $4$-manifolds with holonomy algebra, for the canonical Hermitian connection, of dimension at most one. We show how Riemannian $4$-manifolds admitting five orthonormal 
symplectic forms fit therein and classify them. In this set-up we also fully describe almost K\"ahler $4$-manifolds.
\end{abstract}

\maketitle
\vspace{-10mm}
\tableofcontents
\vspace{-10mm}
\section{Introduction}

\noindent 
The existence of orthogonal harmonic forms on an oriented Riemannian four-manifold 
$(M^4,g)$ typically encodes relevant properties of the metric. An orthonormal frame of closed $1$-forms, for instance, will flatten $g$.

Also closed, orthonormal $2$-forms impose severe constraints on $(M^{4}, g)$, essentially according to the choices of orientation available. Many cases have been addressed in the literature: couples or triples of this kind, defining the same orientation, were studied in \cite{Salamon:special, Geiges:symplectic-couples, Geiges-G:triples}, whereas \cite{Bande-K:symplectic-pairs} dealt with pairs of oppositely-oriented symplectic forms. 

The aim of this note is to consider smooth Riemannian manifolds $(M^4, g)$ admitting a system of \emph{five}  symplectic forms $\{\omega_k, 1 \leq k \leq 5\}$ such that 
\be
\label{eq:5ff-1}
g(\omega_i, \omega_j)=\delta_{ij} \qq i,j=1,\ldots, 5.
\ee
As we will see in section \ref{sec:2forms}, equation \eqref{eq:5ff-1} is equivalent to considering, on a $4$-manifold $M$, 
a so-called \emph{$5$-frame}, that is  five non-degenerate $2$-forms  satisfying
\be
\label{eq:5ff}
\omega_i \wedge \omega_j=\pm\delta_{ij}\,\o_{5}\wedge \o_{5} \qq i,j=1,\ldots, 4
\ee
at each point of $M$. A \emph{closed $5$-frame} is a $5$-frame of symplectic forms. It is known that if $\omega_1,\ldots,\o_6$ is an orthonormal frame of closed $2$-forms then the metric $g$ must be flat, a case we will exclude a priori.

It is easy to see that, up to a re-ordering, three of the $2$-forms are anti-selfdual, and furnish a hyperK\"ahler (hence Ricci-flat) metric $g$. The $2$-forms remaining in the frame are selfdual and give a holomorphic-symplectic form: this in turn yields a 
complex structure $I$, for which $g$ is necessarily Hermitian. 

Our first result is
\begin{teo} \label{story1ends}
Let $(M, g)$ be a non-flat Riemannian $4$-manifold equipped with 
a closed $5$-frame. Then 
\begin{itemize}
\item[(i)] there exists a tri-holomorphic Killing vector field for the hyperK\"ahler structure;
\item[(ii)] $(M, g)$ is  locally isometric to  $\R^{+}\times \textsl{Nil}^{\,3}$ 
equipped with metric 
$$
dt^{2}+ (\tfrac 23 t)^{3/2}(\sigma_{1}^{2}+\sigma_{2}^{2})+ (\tfrac 23 t)^{-2/3}
\sigma_{3}^2,
$$
 where $\{\sigma_{i}\}$ is a basis of left-invariant one-forms on the Heisenberg group $\textsl{Nil}^{\,3}$ satisfying $d\sigma_{1}=d\sigma_{2}=0, d\sigma_{3}=\sigma_{1}\wedge \sigma_{2}$.
\end{itemize}
Moreover the $5$-frame is unique, up a constant rotation in $O(3) \times U(1)$.
\end{teo}
We are thus dealing with Hermitian, selfdual, Ricci-flat  surfaces, classified in \cite{Apostolov-G:sd-EH}. 
To prove theorem \ref{story1ends} we need to determine, first, which metrics in  \cite{Apostolov-G:sd-EH} admit a holomorphic-symplectic form, and then determine that form explicitly.  

The technique used for the proof is indicative of another point of view for looking at closed $5$-frames, namely that of holonomy. We prove that the curvature tensor $\tiR$ of the canonical Hermitian 
connection of $(g,I)$, cf. section \ref{sec:curvature}, is algebraically defined by the structure's Lee and K\"ahler forms (proposition \ref{hol-vhk}). This is used to show that the 
holomorphic Killing field $X$, coming from the Goldberg-Sachs theorem \cite{Nur,ag} for Einstein-Hermitian metrics, is actually tri-ho\-lo\-mor\-phic for the hyperK\"ahler structure. This proves part (i), and as a 
consequence, we know that $g$ is essentially described by the celebrated Gibbons-Hawking Ansatz. 

By changing the sign of $I$ on the distribution spanned by $X, IX$ we obtain a K\"ahler structure $(g,J)$ with `negative' orientation. This allows us to use  \cite[Thm 1]{Apostolov-G:sd-EH}, and eventually the metric $g$ 
is given as in (ii). Note this is a cohomogeneity-one Bianchi metric of type II (see \cite{Tod:Bianchi} for details).\\

The holonomy approach to close $5$-frames leaves much space (section \ref{sec:small-holo}) for investigating almost Hermitian $4$-manifolds $(M,g,I)$ for which the holonomy 
group of the canonical Hermitian connection is `small': since this has dimension bounded by four, small will mean of dimension zero or one. Then there are three possibilities for the two-form corresponding 
to the holonomy generator. It can vanish identically, in which case we prove $g$ must be flat (theorem \ref{flat-gen}), generalising earlier results \cite{Bal,Bal2}. It can be proportional to the K\"ahler form $\o_I$ of $(g,I)$, and this is 
precisely the set-up of selfdual Ricci-flat $4$-manifolds (proposition \ref{hol-vhk}). 
In the third case the holonomy generator has a component orthogonal to $\o_I$ that defines a K\"ahler structure reversing the orientation. In this situation almost K\"ahler structures 
are explicitly classified (theorem \ref{t-ak}). The corresponding examples are build deforming the product of $\bb{R}^2$ with a Riemann surface, in the spirit of \cite{Apostolov-AD:integrability}, \cite{Chiossi-Nagy}.

\section{Two-forms on $4$-manifolds} \label{sec:2forms}
\noindent
Consider a smooth oriented Riemannian $4$-manifold $(M,g)$. The Hodge star operator $\star$ acting on the bundle of two-forms $\Lambda^2$ 
is an involution, with the rank-three subbundles of selfdual and anti-selfdual  $2$-forms $\La^{\pm}=\ker (\star \mp \textsl{Id}_{\Lambda^2})$ as eigenspaces. The resulting decomposition 
\be
\label{eq:A/SDforms}
\La^2=\La^+\oplus \La^-
\ee
is loosely speaking the `adjoint' version of the fact that  $\so(4)=\so(3)\oplus \so(3)$ is semisimple.
 
Now let $I$ in $T^*M\otimes TM$ be an orthogonal almost complex structure, that is 
$$
I^2=-\textsl{Id}_{TM},\qq g(I \cdot,I \cdot)=g (\cdot, \cdot).$$ 
The K\"ahler form $\o_{I}=g(I \cdot, \cdot)$ is non-degenerate at any point of $M$, and induces, by decree, the positive orientation: $\o_{I}^2=\textsl{vol}(g)$. 
The almost complex structure $I$ extends to the exterior algebra by 
\begin{equation*}
(I \alpha)(X_1, \dots, X_p)=\alpha(I X_1, \dots, I X_p),
\end{equation*}
where $\alpha$ is a $p$-form on $M$ and $X_1, \dots, X_p$ belong to $TM$. We shall work with real-valued forms, unless specified otherwise. Notation-wise,  
$()^{\flat}:TM \to \Lambda^1M$ is the isomorphism induced by the metric, with inverse $()^{\sharp}$.

The bundle of real $2$-forms also decomposes under $\U(2)$ as 
\be\label{eq:inv-2forms}
\La^2=\la^{1,1} \oplus \la^2,   
\ee
where $\la^{1,1}$ denotes $I$-invariant two-forms, and  $ \la^2$ anti-invariant ones. If $\langle \omega_{I}\rangle$ is the real line through $\omega_{I}$ we further have $\la^{1,1}=\lan \o_{I}\ran \oplus\la^{1,1}_{0}$, where the 
space of primitive $(1,1)$-forms $\lambda^{1,1}_0$ is the orthogonal complement to $\langle \omega_{I}\rangle$ in $\lambda^{1,1}$. 
The real rank-two bundle $\lambda^2$ has a complex structure $\mathcal{I}\alpha=\alpha(I \cdot, \cdot)$ that makes it a complex line bundle isomorphic to the canonical bundle of $(M,I)$. 
So in presence of an almost Hermitian structure, $\Lambda^{+}$ splits as   
$$ 
\La^+=\lan \o_{I}\ran  \oplus  \la^2$$
under $U(1)$ and comparing \eqref{eq:inv-2forms} and \eqref{eq:A/SDforms} leads to $ \La^-=\la^{1,1}_{0}.$

Slightly changing the point of view, we briefly recall how to recover a conformal structure in dimension four from rank three subbundles of two-forms.
Consider a smooth oriented manifold $M$ of real dimension four with volume form $\textsl{vol}(M)$. The bundle $\La^2 M$ possesses a non-degenerate bilinear form $q$ given by 
$$
\alpha\w \beta=q(\alpha,\beta)\textsl{vol}(M)
$$
whenever $\alpha, \beta$ belong to $\Lambda^2M$. Any subbundle $E$ of $\Lambda^2M$ of rank three and maximal, in the sense that $q_{\vert E}$ is positive,  determines a unique conformal structure on $M$ such 
that $E=\Lambda^{+}M$ \cite{Salamon:instantons}. The proof of this descends from the fact that $q$ has signature $(3,3)$ and that 
at any point $x$ of $M$ the set of maximal subspaces of $\Lambda^2T_xM$ is parametrised by 
$$ \frac{GL(4,\R)}{CO(4)} \iso \frac{\SL(4,\C)}{\SO(4)} \iso \frac{\SO_0(3,3)}{\SO(3)\times\SO(3)}.
$$
Here $CO(4)=\R^+\times \SO(4)$ is the conformal group. In particular 
\begin{lemma} \label{lema1}
A closed $5$-frame on a smooth $4$-manifold determines
\begin{itemize}
\item[(i)] a Riemannian metric $g$ such that
\be
\label{eq:convention}
\span\{ \omega_1,\omega_2,\omega_3\}=\Lambda^{-}, \ \{\omega_4, \omega_5\} \subset \Lambda^{+},
\ee
up to re-ordering;
\item[(ii)] an orthogonal complex structure $I$ defined by $\o_5=\o_4(I \cdot, \cdot)$.
\end{itemize}
\end{lemma}
\begin{proof}
(i) The $\omega_i$ are linearly independent as the metric $q$ is neutral; for the same reason 
in \eqref{eq:5ff} there are effectively $3$ minus signs and $2$ pluses,\footnote{\ or the other way around, but we will suppose three minuses.} eg $\o_{1}^{2}=\o_{2}^{2}=\o_{3}^{2}$,  
and in the chosen conformal class there exists a unique Riemannian metric $g$ such that $\{\omega_4, \omega_5 \}$ are orthonormal, so that $-\o_{1}^{2}=\o_{4}^{2}=\o_{5}^{2}$. \\
(ii) The closure of $\omega_4$ and $\omega_5$ implies automatically that $I$ is integrable \cite{Salamon:special}.
\end{proof}
The convention throughout this note will be that of \eqref{eq:convention}. Then the fundamental form $\o_I=g(I \cdot, \cdot)\in\Lambda^{+}$ completes the $5$-frame to an orthonormal basis of $\Lambda^2$. When the metric $g$ is flat the 
forms $\omega_4, \o_5$ must be parallel for the Levi-Civita connection of $g$ by a result of \cite{Armstrong:4-aK}, so from now we consider non-flat $5$-frames, that is $g$ will be assumed not flat.

Since the triple $\{\o_1, \o_2, \o_3\}$ defines a hyperK\"ahler structure, the metric $g$ is Ricci-flat and selfdual.  By lemma \ref{lema1} the classification of 
closed $5$-frames amounts to that of Ricci-flat, selfdual Hermitian structures $(g,I)$ equipped with 
a complex symplectic structure, that is a closed, constant-length two-form $\o_4+i\o_5$ in $\Lambda^{0,2}_IM=\Lambda^2(M,\bb{C}) \cap \ker(\mathcal{I}+2i)$.

The article \cite{Apostolov-G:sd-EH} contains the complete local-structure theory for Ricci-flat, selfdual Hermitian surfaces; to locate closed $5$-frames in that classification we will set up, in the next section, an equivalent curvature description. Before 
doing so we remind that Gibbons and Hawking \cite{Gibbons-H:gravitational} have generated, locally, all hyperK\"ahler $4$-manifolds admitting a tri-holomorphic Killing vector field using Laplace's equation in Euclidean three-space. We outline 
below their construction to show how closed $5$-frames fit therein. 

Take a hyperK\"ahler $4$-manifold $(M^4,g,J_{1},J_{2},J_{3})$ with a vector field $X$ such that $L_XJ_i=0, 1 \leq  i \leq  3$ and $L_Xg=0$. Choose a local 
system of co-ordinates $(u,x,y,z)$ on $M$ with $X=\frac{d}{du}$ and where 
 $x,y,z$, given by $X \lrcorner\, \o_{J_1}=dx$, etc., are the momentum maps. In these co-ordinates the metric reads 
$$
g=U(dx^2+dy^2+dz^2)+U^{-1} (du+\Theta)^2,
$$
where  $U(x,y,z)=\Vert X \Vert^{-2}$ is defined on some domain in $\bb{R}^3$ and the connection one-form $\Theta$ is invariant under $X$ and such that $\Theta(X)=0$. 
The fundamental forms of $(g,J_k), 1 \leq  i \leq  3$, 
$$
\o_{J_{1}}=Udy dz+dx(du+\Theta),\q 
\o_{J_{2}}= Udx dy+dz(du+\Theta),\q 
\o_{J_{3}}= Udz dx+dy(du+\Theta)
$$
 are closed if and only if  $\Theta$ satisfies the monopole equation $d\Theta=\star_{\bb{R}^3} dU$. In particular $U$ is harmonic on some open region of 
$\bb{R}^3$, and conversely such a function completely determines the geometry as explained above. 

Moreover the (non-necessarily closed) forms 
$$
\o_{I_{1}}=Udy dz-dx(du+\Theta),\q 
\o_{I_{2}}= Udx dy-dz(du+\Theta),\q 
\o_{I_{3}}= Udz dx-dy(du+\Theta)
$$
are orthonormal and yield a trivialisation of $\Lambda^{+}$.
\begin{ex} \label{ex:GH}
Imposing the forms $\o_{I_1}$ and $\o_{I_2}$ be closed forces $U_x=U_z=0,$ so $U=ay+b$ for real constants $a,b$; one can explicitly take 
$\Theta=\frac{a}{2}(zdx-xdz)$. In this situation $\o_{I_1}$ and $\o_{I_2}$ build, together with $\o_{J_k}, 1 \leq  k \leq  3$, a closed $5$-frame, which is not flat 
for $a \neq 0$ since $d\o_{I_3}=2adxdydz$ does not vanish. 
\end{ex}
Roughly speaking, theorem \ref{story1ends} explains why this example is no coincidence.\smallbreak
\section{The curvature of the canonical connection} 
\label{sec:curvature}

\noindent
In order to characterise closed $5$-frames in terms of curvature we need some facts from almost Hermitian geometry;  this will serve us beyond the $5$-frame set-up as well, so the presentation will be general. 

Let $(M^4,g,I)$ be almost Hermitian. The Levi-Civita connection $\nabla$ of $g$ defines the so-called intrinsic torsion 
of $(g,I)$ 
$$
\eta=\frac{1}{2}(\nabla I)I \in \Lambda^1 \otimes \lambda^2.
$$ 
Its knowledge is the main tool to capture, both algebraically and not, the geometry of almost Hermitian manifolds:  indeed, the components of $\eta$ inside the irreducible $\U(2)$-modules into 
which $\Lambda^1 \otimes \lambda^2$ decomposes determine the type and features of the structure under scrutiny (eg K\"ahler, Hermitian, conformally K\"ahler and so on). When indexing a differential form with a vector we shall mean  $\beta_X=X\hook\beta=\beta(X,\cdot,\ldots,\cdot)$, and in particular $\eta_X=\frac{1}{2}(\nabla_X I)I.$

The canonical connection
$$
\tnabla=\nabla+\eta
$$ 
of the almost Hermitian structure $(g,I)$  is a linear connection that preserves Riemannian and almost complex structures, $\tnabla g=0$ and $\tnabla I=0$, hence it is both  metric and Hermitian.  
It coincides with the Chern connection (see \cite{Gauduchon2}) if $I$ is integrable.
Since the torsion tensor $T$ of $\tnabla$ is given by 
$$T_XY=\eta_XY-\eta_YX$$
 for any tangent vectors $X,Y$, we have $\eta=0$  if and only if $(g,I)$ is K\"ahler. 
The canonical Hermitian connection naturally induces an exterior derivative on bundle-valued differential forms. For instance if $\alpha $ belongs to $\Lambda^1(M, \lambda^2)$, we have $d^{\tnabla}\alpha(X,Y)=
(\tnabla_X\alpha)Y-(\tnabla_Y\alpha)X$ for any $X,Y$ in $TM$. On ordinary differential forms 
 one defines $d^{\tnabla} : \Lambda M \to \Lambda M$ in analogy with the usual exterior derivative, that is $d^{\tnabla}= \sum \limits_{i=1}^4 e_i \wedge \tnabla_{e_i}$
where $\{e_i, 1\leq i \leq 4\}$ is an orthonormal basis of each tangent space. If the action of $T$ 
on some $1$-form $\alpha$ is defined by $(T\alpha)(X,Y)=\alpha(T_XY)$ for any $X,Y$ in $TM$, we may compare differentials
$$ 
d^{\tnabla} \alpha=d\alpha-T\alpha.
$$

\noindent
Given a local gauge $I_1$, that is a locally-defined orthogonal complex structure 
such that $I_1 I+I I_1=0$, we define  $I_2=I_{1} I$, and  write  
\begin{equation}  \label{gtot}
\nabla I=a \otimes I_2+c \otimes I_1,\q\textrm{or equivalently}\q
2\eta=-a \otimes I_1+c \otimes I_2,
\end{equation}
for  local $1$-forms $a,c$  on $M$.
The curvature tensor $\tiR \in \Lambda^2 \otimes\la^{1,1}$ of the canonical connection, defined by $\tiR(X,Y)=-[\tnabla_X ,\tnabla_Y]+\tnabla_{[X,Y]}$, $X,Y$ in $TM$, has in general not all  of the symmetries 
enjoyed by the Riemannian counterpart $R$. 
It fails to be symmetric in pairs, and does not satisfy the first Bianchi identity, due to the terms involving the intrinsic torsion in 
\begin{equation} \label{comp-curv}
\tiR(X,Y)=R(X,Y)- d^{\tnabla}\eta(X,Y)+[\eta_X,\eta_Y]-\eta_{T_XY}
\end{equation}
for any $X,Y$  in $TM$, see eg \cite{Cleyton-S:intrinsic}. The algebraic summands above can be computed locally from \eqref{gtot}; in particular
\begin{equation} \label{bra-tor}
[\eta_X, \eta_Y]=\frac{1}{2}\Phi(X,Y) I
\end{equation}
for all $X,Y$ in $TM$, where $\Phi=a \w c$. The first Chern form $\widetilde{\gamma}_1$ of the canonical connection is the $2$-form defined by 
$$
\widetilde{\gamma}_1(X,Y)=\langle \tiR(X,Y), \omega_{I}\rangle
$$
for any $X,Y$ in $TM$, where the brackets $\langle \cdot, \cdot\rangle$ denote the inner product on forms. The differential Bianchi identity forces it to be 
closed, $d\widetilde{\gamma}_1=0$, and moreover $\frac{1}{2\pi} \widetilde{\gamma}_1$ is a de Rham 
representative for $c_1(M,I)$. Splitting  the Ricci tensor $\Ric=\Ric^{\prime}+\Ric^{\prime \prime}$ into  invariant and anti-invariant parts under $I$, and taking the scalar product with $\omega_{I}$ in \eqref{comp-curv}, yields  
\begin{equation} \label{28}
\widetilde{\gamma}_1=\rho^{I}+W^{+}\omega_{I}+\Phi-\frac{\s}{6}\omega_{I}.
\end{equation}
Here  $\rho^{I}=\langle \Ric^{\prime} I \cdot, \cdot \rangle \in\lambda^{1,1}$, $s$ is the scalar curvature of $g$ and $W^\pm=\tfrac 12(W\pm W \star)$ are the positive and negative halves of the Weyl curvature considered as a bundle-valued $2$-form
\begin{equation*}
W=W^+ + W^-,
\end{equation*}
reflecting \eqref{eq:A/SDforms}. The  $4$-manifold is called selfdual or anti-selfdual according to whether $W^-=0$ or $W^+=0$.

One may also compute the first Chern form locally, by expanding the covariant derivative of a local gauge $I_1$:
\begin{equation}  \label{b-out} 
\begin{split}
\tnabla I_1= -b \otimes I_2, \ \tnabla I_2= b \otimes I_1
\end{split}
\end{equation}
where $b$ is a local $1$-form on $M$, which implies 
\begin{equation} \label{Chern-loc}
\widetilde{\gamma}_1=-db.
\end{equation}
Expression \eqref{comp-curv} for  $\tiR$ simplifies considerably if one uses the Weyl tensor. Let $\Ric_0$ denote the trace-free component of the Ricci tensor and $h=\frac{1}{2}(\Ric_0+\frac{\s}{12}g)$ be the reduced Ricci tensor of $g$. 
Then $R=W+h \wedge g$, where 
$ (h \wedge g)(X,Y)=hX \wedge Y+X \wedge hY$ for any $X,Y$ in $TM$. The latter  can be written as $(h \wedge g)F=\{F,h\}$ for any $F$ in $\Lambda^2 \cong \mathfrak{so}(TM)$, where $\{\cdot, \cdot\}$ denotes the anti-commutator. 
Due to the isomorphism 
 $\textsl{Sym}_0^2 \iso \Lambda^{+} \otimes \Lambda^{-}$ described by the map $S \mapsto S^{-}$, where $S^{-}(F)=\{S,F\}^{-}$, $F$ in $\Lambda^{2}$, we also know that  $\{\textsl{Sym}^2_0, \Lambda^{\pm}\} \subseteq \Lambda^{\mp}$.\\
 
The next lemma generalises a statement of  \cite{Gauduchon:torsion}. 
\begin{lema} \label{curv-C-2}
On an almost Hermitian manifold $(M^4,g,I)$ the curvature of the canonical connection can be decomposed as
\begin{equation*}
\tiR=W^{-}+\frac{\s}{12}\textsl{Id}_{\Lambda^{-}}+\frac{1}{2}\Ric_0^{-}+\frac{1}{2}\,\widetilde{\gamma}_1 \otimes \omega_{I}.
\end{equation*}
\end{lema}
\begin{proof}
Expanding \eqref{comp-curv}, and using \eqref{bra-tor} along the way, we obtain 
\begin{equation*}
\begin{split}
\tiR(X,Y)=&\ W^{-}(X,Y)+W^{+}(X,Y)+hX \wedge Y+X \wedge hY\\
&-d^{\tnabla}\eta(X,Y)+\frac{1}{2}\Phi(X,Y) \omega_{I}-\eta_{T_XY}
\end{split}
\end{equation*} 
for any $X,Y$ in $TM$. Now, since $d^{\tnabla}\eta(X,Y)+\eta_{T_XY}$ belongs to $\lambda^2$ and  $\tiR$ lives in $\Lambda^2 \otimes \lambda^{1,1}$, by taking into account that $W^{+}$ only 
acts on $\langle \omega_{I}\rangle \oplus \lambda^2$, we can  project onto invariant $2$-forms and infer  
\begin{equation*}
\tiR=W^{-}+(h \w g)_{\lambda^{1,1}}+\frac{1}{2}(W^{+}\omega_{I}+\Phi) \otimes \omega_{I}.
\end{equation*}
From $\lambda^{1,1}=\Lambda^{-}\oplus \langle \omega_{I} \rangle$ we further get 
$(h \w g)_{\lambda^{1,1}}=(h \w g)^{-}+\frac{1}{2}\langle \{h,I\} \cdot, \cdot \rangle \otimes \omega_{I}$, 
and the claim follows by definition of $h$ and equation \eqref{28}. 
\end{proof}
\subsection{Elements of Hermitian geometry}
We now specialise the facts above to $(M^4,g,I)$ being Hermitian. Equivalently
$\eta\in \lambda^{1,1} \otimes \Lambda^1$, which 
in a local gauge $I_1$ means that the $1$-forms $c, a$ of \eqref{gtot} 
satisfy 
\begin{equation} \label{h-gau}
c=-Ia, 
\q \theta=2I_1a
\end{equation}
where the Lee form $\theta$ is defined by $d\o_I=\theta \w \o_I$. A simple computation yields 
\begin{equation*} 
\eta_U=\frac{1}{4}(U^{\flat} \wedge \theta+(I U)^{\flat} \wedge I \theta)
\end{equation*}
for any $U$ in $TM$. Consequently $\eta_{\z}=\eta_{I \z}=0$, where $\z=\theta^{\sharp}$.
It follows easily that 
\begin{equation} \label{Phi}
\Phi=\frac{1}{4}(\theta \w I\theta+ \vert \theta \vert^2 \omega_I).
\end{equation}
Let $\kappa=3\langle W^{+} \omega_{I}, \omega_{I}\rangle$ be the conformal scalar curvature, which differs from the usual scalar curvature by
\begin{equation} \label{ks}
\kappa-\s=-3d^{\star}\theta-\frac{3}{2} \vert \theta \vert^2,
\end{equation}
see \cite{Gauduchon:torsion}. Given a one-form $\alpha$  we denote by $d^{\pm} \alpha$ the components of $d\alpha$ in $\Lambda^{\pm}$ respectively, so that $d\alpha=d^{-}\alpha+d^{+}\alpha$. An important property \cite {ag} of the positive Weyl tensor of the Hermitian structure is 
\begin{equation*} \label{w+}
W^{+}=\frac{\kappa}{4}\left(\frac{1}{2}\omega_{I} \otimes \omega_{I}-\frac{1}{3}{\sl Id}_{\vert \Lambda^{+}}\right)-\frac{1}{4}\Psi  \otimes \omega_{I}-\frac{1}{4}\omega_{I} \otimes \Psi
\end{equation*}
where $\Psi=d^{+} \theta (I \cdot, \cdot)$ belongs to $\lambda_I^2$. In particular $W^{+} \omega_{I}=\frac{\kappa}{6}\omega_{I}-\frac{1}{2}d^{+}\theta(I \cdot, \cdot)$, hence \eqref{28} updates, 
with the aid of \eqref{Phi} and \eqref{ks},  to 
\begin{equation} \label{29} 
\widetilde{\gamma}_1=\rho^I-\frac{1}{2}(d^{\star} \theta) \omega_I+\frac{1}{4} \theta \w I \theta-\frac{1}{2}\Psi.
\end{equation}
It is well known that $d^{+} \theta=0$ is equivalent to demanding $W^{+}$ to be degenerate, which is a short way of saying that $W^{+}$ has a double eigenvalue. 
\section{Proof of theorem \ref{story1ends}} 
\label{sec:proof}

\noindent
As mentioned earlier, closed $5$-frames are equivalently described by Ricci-flat, selfdual Hermitian manifolds $(M^4,g,I)$ that admit a holomorphic-symplectic structure compatible with the complex orientation.

The crucial observation is the following characterisation of such structures by means of the curvature of their canonical Hermitian connection. 
This approach will be taken up in the next section, in a more general situation.
\begin{pro}\label{prop42} 
A Hermitian manifold $(M^4,g,I)$ admits, around each point, a closed $5$-frame if and only if 
$$
\tiR=-\frac{1}{4}d(I \theta) \otimes \omega_{I}.
$$ 
\end{pro}
\begin{proof} 
Lemma \ref{curv-C-2}.guarantees that the metric $g$ is Ricci-flat and selfdual if and only if $\tiR= \frac{1}{2}\,\widetilde{\gamma}_1 \otimes \omega_{I}$. Equivalently, there 
exists a $g$-compatible hyperk\"ahler structure $\{\o_1, \o_2, \o_3 \}$ spanning $\Lambda^{-}$ around each point in $M$. There remains to show that the existence of an orthonormal pair 
$\o_4, \o_5$ of closed forms in $\lambda^2_I$ is the same as $\widetilde{\gamma}_1=-\frac{1}{2}d(I \theta)$.

Suppose  $I_1$ is a local gauge for $(g,I)$  in the notation of \eqref{b-out}. Then by writing $\nabla I_2=-a \otimes \omega_{I}+b \otimes \omega_{I_1}$ the closure of $\omega_{I_2}$ is equivalent 
to  $ -a \wedge \omega_{I}+b \wedge \omega_{I_1}=0$. But equation \eqref{h-gau} says $a \wedge \omega_{I}=-\frac{1}{2}I_1 \theta \wedge \omega_{I}= \frac{1}{2}I \theta \wedge \omega_{I_2}$, hence $b=\frac{1}{2}I \theta $. 

Now, assume first that $ \tiR=-\frac{1}{4}d(I \theta) \otimes \omega_{I}$, so that $\widetilde{\gamma}_1=-\frac{1}{2}d(I \theta)$. A straightforward computation shows that the Hermitian connection $D=\tnabla-\frac{1}{4}I \theta \otimes I$ has zero 
curvature. Take a local orthonormal frame $e_1,\ e_2=I e_1,\ e_3,\ e_4=I e_3$
such that $De_k=0, 1 \leq k \leq 4$. Then $\omega_{I}=e^{12}+e^{34}$ ($e^{12}$ meaning $e^1\w e^2$), and the other selfdual forms 
$e^{14}+e^{23}$, $e^{13}+e^{42}$ can be written as $g(I_1 \cdot, \cdot)$, $g(I_2 \cdot, \cdot)$ respectively, with  $I_2=I_1I$. Since $I_2$ is $D$-parallel we  have  
$\tnabla I_2=\frac{1}{2}I \theta \otimes I_1$. Equation \eqref{b-out} gives $b=\frac{1}{2}I\theta$, hence $\omega_{I_2}$ is closed, and so is $\omega_{I_1}$  \cite{Salamon:special}.
The anti-selfdual forms $e^{12}-e^{34}, e^{13}-e^{42}, e^{23}-e^{14}$
are $D$-parallel by construction. But they are $\nabla$-parallel as well, for $D-\nabla$ belongs to $\Lambda^1 \otimes \Lambda^{+}$, and selfdual and anti-selfdual forms commute. The construction of the $5$-frame is thus complete.

Vice versa, assume that 
$\omega_4=g(I_1 \cdot, \cdot), \omega_5=g(I_2 \cdot, \cdot)$ with $I=I_2I_1$ are closed forms in $\lambda^2_I$. The above argument  gives $b=\frac{1}{2}I\theta$, hence again
$\widetilde{\gamma}_1=-\frac{1}{2}d(I \theta)$ by \eqref{Chern-loc}.  
\end{proof}
\noindent
If in addition $M$ is simply connected the $5$-frame is global. 
At this point we invoke the theory of selfdual Hermitian-Einstein manifolds, as presented in \cite{Apostolov-G:sd-EH}.  Well-known facts are collected in the following 
\begin{pro}[\cite{ag, Derdzinski:SDKE-4-mfds, Nur}] \label{class-5forms} 
Let $(M^4,g,I)$ be Hermitian, Ricci-flat and selfdual, 
but not flat. Then
\begin{itemize}
\item[(i)] $\omega_{I}$ is an eigenform of $W^{+}$, ie  $W^{+} \omega_{I}=\frac{\kappa}{6}\, \omega_{I}$;\vspace{3pt}
\item[(ii)] the conformal scalar curvature $\kappa$ and Lee form $\theta$ satisfy $\kappa  \theta+\frac{2}{3}d \kappa =0$,  and $(\kappa^{\frac{2}{3}}g, I)$ is K\"ahler;\vspace{3pt}
\item[(iii)] $X=I \grad (\kappa^{-\frac{1}{3}})$ is a Hamiltonian Killing vector field;\vspace{3pt}
\item[(iv)] $d^{+}X^{\flat}=-\frac{1}{12}\kappa^{\frac{2}{3}}\omega_{I}$. 
\end{itemize}
\end{pro}
\noindent
Conditions (i) and (ii) are equivalent on any compact, not necessarily Einstein, Hermitian complex surface \cite{Boyer:conformal-duality,Apostolov-G:sd-EH}. Part (i) holds for compact selfdual Hermitian surfaces \cite{adm} as well. 
\smallbreak

{\it Proof of theorem \ref{story1ends}.} 
Since a closed $5$-frame induces a selfdual, Ricci-flat Hermitian metric, in order to use the  classification of \cite{Apostolov-G:sd-EH} we will show first 
that the Killing field $X$ above is tri-holomorphic for the local hyperK\"ahler structure. Examples in \cite{Ward:class} confirm that this is not true in general. 

(i - ii) The conformal scalar curvature $\kappa$ is nowhere zero, otherwise the metric $g$ would be flat. Ricci flatness implies  
$d\theta=0$ by (ii) in the proposition above. By proposition \ref{prop42} we have  $\widetilde{\gamma}_1=-\frac{1}{2}d(I \theta)$, hence \eqref{29} implies 
$$-\frac{1}{2}d(I \theta)=\frac{1}{4}\theta \wedge I \theta-\frac{d^{\star} \theta}{2}\omega_{I}.$$ 
Using proposition \ref{class-5forms} (ii), and the comparison formula \eqref{ks}, we get $d(\kappa^{-\frac{1}{3}}I\theta)=-\kappa^{-\frac{1}{3}}(\frac{\kappa}{3}+\frac{\vert \theta \vert^2}{2}) \omega_{I}$. Since 
$X^{\flat}=\frac{\kappa^{-\frac{1}{3}}}{2}I \theta$  we obtain  $d^{-}X^{\flat}=0$; it follows that the Killing  vector field $X$ is tri-holomorphic with respect to  
the local hyperK\"ahler structure.

At the same time,
by comparing with proposition \ref{class-5forms} (iv), it follows that 
$\vert X \vert^2+\tfrac{1}{12}\kappa^{\frac{1}{3}}=0$ or equivalently 
$ \vert \theta \vert^2=-\frac{\kappa}{3}.$
Then $d\ln \vert \theta \vert=-\frac{3}{4}\theta$ belongs to the distribution $\D$ spanned by $\theta^{\sharp}, I\theta^{\sharp}$; this means, according to 
\cite[theorem 1]{Apostolov-G:sd-EH}, that the orthogonal almost complex $J$, obtained by reversing the sign of $I$ along $\D$, is integrable. Its fundamental form 
$$\o_J=\o_I+2 \vert \theta \vert^{-2} \theta \w I \theta$$ 
belongs to $\Lambda^{-}$ and it is closed; indeed $d\theta=0$, and since $X$ is tri-holomorphic,  $dX^{\flat}=-\frac{1}{12}\kappa^{\frac{2}{3}}\omega_{I}$ by proposition \ref{class-5forms}, so 
\begin{equation*}
d\o_J=\theta \wedge \o_I-2 \theta \w d(\vert \theta \vert^{-2}I \theta)=\theta \w \o_I+12 \theta \wedge d(\kappa^{-\frac{2}{3}} X^{\flat})=
\theta \w \o_I+12\kappa^{-\frac{2}{3}} \theta \w dX^{\flat}=0.
\end{equation*}
Therefore $(g,J)$ is a K\"ahler structure. Theorem 1 of \cite{Apostolov-G:sd-EH},  case b1) of its proof to be precise, warrants that selfdual 
Ricci-flat Hermitian $4$-manifolds with $X$ tri-holomorphic and $(g,J)$ K\"ahler reduce to the Gibbons-Hawking Ansatz with $U=ay+b$ for constants $a, b$. Since $g$ is not flat we can take $a\not=0$, and without of loss of generality let 
$a=1$ by rescaling the metric, and  $b=0$. In this way $g$ is of the form claimed, with
$t = \frac 23 y^{3/2},\ \sigma_{1}=dz,\  \sigma_{2}=dx,\ \sigma_{3}=du+\Theta$
in the notation of example \ref{ex:GH}.

As for the theorem's last statement, the anti-selfdual part of a closed $5$-frame is unique up to an $O(3)$-rotation. Let now $\omega_{4}^{\prime}, \omega_{5}^{\prime} $ be orthonormal and closed in $\Lambda^{+}$. Up to a sign 
they determine \cite{Salamon:special} the 
same orthogonal complex structure as $\omega_4, \omega_5$, since $W^{+}$ is  degenerate and never zero and so they belong to  
$\lambda^2_{I}$. Then $\omega_4^{\prime}+i\omega_5^{\prime}=f(\omega_4+i\omega_5)$, with $f:M \to U(1)$  holomorphic with respect to $I$ due to the closure of the forms, and therefore constant.
\qed
\smallbreak

At this juncture a few comments are in order. First, a non-flat closed $5$-frame is incompatible with the manifold being compact. 
In fact, if the induced metric were even only complete, $X$ would become a global tri-holomorphic Killing vector field, in contradiction to \cite[theorem 1 (iii)]{bielawski}.
Secondly, theorem \ref{story1ends} can be considered as a local $4$-dimensional analogue, for two-forms, of the following result:\smallbreak

\noindent {\bf Theorem \cite{NV}}. 
{\it Let $(M,g)$ be a compact Riemannian $n$-manifold with $b_1(M)=n-1$ and such that every harmonic $1$-form has constant length. Then $(M,g)$ is a quotient of a nilpotent Lie group with $1$-dimensional centre, equipped with a left-invariant metric.}\smallbreak

\section{Small holonomy and further examples}\label{sec:small-holo}

\noindent
The proof of theorem \ref{story1ends} suggests a wider perspective should be considered, namely that of almost Hermitian $4$-manifolds with small holonomy.

Given $(M^4,g,I)$ almost Hermitian, consider the holonomy algebra $\widetilde{\mathfrak{hol}} \subseteq \mathfrak{u}(2)$ of the canonical connection at a given point of $M$, and assume it at most $1$-dimensional. Then any generator of $\widetilde{\mathfrak{hol}}$ must be invariant under parallel transport by  $\tnabla$, so it must extend to an element $F$ of $\Lambda^2$ such that $ \tnabla F=0$. Since the curvature tensor $\tiR$ takes its values in $\widetilde{\mathfrak{hol}}$ we can write 
\begin{equation*}
\tiR=\gamma \otimes F
\end{equation*}
for some two-form $\gamma$ on $M$. As $\tnabla$ is Hermitian $F$ must have type $(1,1)$, hence we can split 
$$ F=F_0+\alpha \omega_{I},
$$
where $F_0$ is in $\lambda^{1,1}_0$ and $\alpha$ a real number.

Three possible scenarios unfold before us: the entire curvature $\tiR$ vanishes,  $F_{0}$ is zero or $F_{0}$ is non-zero.
\subsection{The flat case}
We begin with the simplest situation, in which the almost Hermitian manifold $(M^4,g,I)$ has 
${\tiR = 0}$ everywhere. 

\begin{teo} \label{flat-gen}Let $(M^4,g,I)$ be almost Hermitian and such that $\tiR=0$. Then
\begin{itemize}
\item[(i)] the metric $g$ is flat;
\item[(ii)] if $(\sigma_1,\sigma_2,\sigma_3)$ is a $\nabla$-parallel orthonormal basis of selfdual forms, 
\begin{equation*} 
\omega_{I}= \sigma_1\cos  \varphi  \cos  \psi + \sigma_2\cos  \varphi \sin \psi +\sigma_3\sin \varphi
\end{equation*}
where $\varphi$ and $\psi$ are locally-defined maps on $M$ subject to $d\psi \wedge d\varphi=0$.
\end{itemize}
\end{teo}
\begin{proof}
(i) By lemma \ref{curv-C-2} the tensors $W^{-}, \Ric$ and $\widetilde{\gamma}_1$ all vanish. Therefore, $W^{+}\omega_{I}=-\Phi$ 
by \eqref{28}, implying the two-form $\Phi$ belongs to $\Lambda^{+}$. But $\Phi$ is decomposable in any local gauge, hence it squares to zero. This means that $\Phi$ vanishes, too:
$$ W^{+}\omega_{I}=0, \q \Phi=0.
$$
Since $g$ is Einstein and $W^{+}$ has zero determinant, \cite[proposition 16.72]{Besse} forces $W^{+}=0$, and $g$ is indeed a flat metric.\\
(ii) A local gauge for $\omega_{I}$ is given by 
\begin{equation*} 
\begin{split}
\omega_{I_1}=& \sigma_1\sin \varphi  \cos  \psi +\sigma_2\sin \varphi \sin \psi- \sigma_3\cos  \varphi \\
\omega_{I_2}=& -\sigma_1\sin \psi + \sigma_2\cos  \psi .
\end{split}
\end{equation*}
A straightforward computation yields $a=-d\varphi, \ c=  \cos  \varphi d\psi, \ b=\sin \varphi d\psi$.
From the proof of part (i), $\tiR=0$ is equivalent to $\Phi=a \wedge c=0$ when $g$ is flat. Therefore $\cos \varphi d\varphi \w d \psi=0$, and we conclude by a density argument. 
\end{proof}
\noindent
In addition, the theorem of Frobenius tells that $\psi=\psi(\varphi)$ is a local function of one variable.
\begin{coro} \label{cor1}
Let $(M^4,g,I)$ be  either Hermitian or almost K\"ahler, with $\tiR=0$. Then $(g,I)$ is a flat K\"ahler structure.
\end{coro}
\begin{proof}
The Hermitian and almost K\"ahler conditions are both characterised in a local gauge by $c=\mp I a$, so the claim follows from $a \wedge c=0$ and theorem \ref{flat-gen} (i). 
\end{proof}
In the almost K\"ahler case the corollary was proved  in \cite{Discala-V:gray}  assuming compactness, albeit differently and for arbitrary dimensions. Similar results can be found in \cite{Bal,Bal2}, again for $M$ compact.
\subsection{The case ${F_{0}=0}$}
This is a very familiar situation as the next observation shows.
\begin{pro} \label{hol-vhk} 
On an almost Hermitian manifold $(M^4,g,I)$ the following  are equivalent:
\begin{itemize}
\item[(i)] the curvature of the canonical connection is generated by the K\"ahler form of $I$: 
\begin{equation} \label{curv-1}
\tiR= \frac{1}{2}\,\widetilde{\gamma}_1 \otimes \omega_{I}
\end{equation}
\item[(ii)]  $\Ric=0$ and $W^{-}=0$.
\end{itemize}
\end{pro}
\begin{proof}
The statement is an immediate consequence of lemma \ref{curv-C-2}.
\end{proof}
Note that on a selfdual, Ricci-flat manifold any positive orthogonal almost complex structure satisfies \eqref{curv-1}. 
\noindent
\subsection{The case ${F_{0}\not=0}$}
Because $F_0$ is $\tnabla $-parallel, it has constant length. By rescaling $\gamma$ if necessary we may parametrise $F_0=\o_{J}=g(J\cdot,\cdot)$ by means of an orthogonal complex structure $J$ with 
orientation opposite to that of $I$. 
\begin{pro}  \label{1-dimhol}
Let $(M^4,g,I)$ be almost Hermitian. The following statements are equivalent:
\begin{itemize}
\item[(i)] the holonomy algebra of the canonical connection is $1$-dimensional, generated by $F$ in $\lambda^{1,1}_I$ with non-vanishing primitive part;
\item[(ii)] $\tnabla$ is not flat and there is a negatively-oriented, $g$-compatible K\"ahler structure $J$ such that $\widetilde{\gamma}_1=\alpha \rho^J$, 
where $\alpha$ is a non-zero real constant. 
\end{itemize}
Either assumption implies
\begin{equation*} \label{curv-clean}
\tiR=\frac{\rho^{J}}{2} \otimes (\alpha \omega_{I}+\omega_J).
\end{equation*}
\end{pro}
\begin{proof}
(i) $\Rightarrow$ (ii) It is clear that $\tnabla J=0$. Since $\eta=\tnabla-\nabla$ belongs to 
$\Lambda^1 \otimes \lambda^2 \subseteq \Lambda^1 \otimes \Lambda^{+}$, it follows that   
$\nabla J=0$, for selfdual and anti-selfdual forms commute. Therefore $(g,J)$ is K\"ahler and compatible with the negative orientation. In particular the 
Ricci tensor of $g$ is $J$-invariant and 
$$W^{-}=\biggl ( \begin{array}{lr}
\frac{s}{6} & 0\\
0 & -\frac{s}{12}\\
\end{array} \biggr ) $$
with respect to  $\Lambda^{-}=\langle \omega_J \rangle \oplus \lambda^2_J$. Equivalently, 
\begin{equation} \label{curv-Ka}
W^{-}+\frac{\s}{12} \textsl{Id}_{\Lambda^{-}}+\frac{1}{2}\Ric_0^{-}=\frac{1}{2}\rho^J \otimes \omega_J. 
\end{equation}
Lemma \ref{curv-C-2} gives then $\tiR=\frac{1}{2}\rho^J \otimes \omega_J+\frac{1}{2} \tilde{\gamma_1} \otimes \omega_{I}$. 
From  $\tiR=\gamma \otimes F$ 
$$ \frac{1}{2}\rho^J \otimes \omega_J+\frac{1}{2} \widetilde{\gamma}_1 \otimes \omega_{I}= \gamma \otimes \omega_J+\alpha \gamma \otimes \omega_{I}
$$
follows, and proves the claim.\\
The implication (ii) $\Rightarrow$ (i) is a direct consequence of \eqref{curv-Ka} and lemma \ref{curv-C-2}, which also prove the final assertion.
\end{proof}
In the rest of this section we determine explicitly the almost K\"ahler structures $(M^4,g,I)$ with $\dim \widetilde{\mathfrak{hol}} \leq 1$. We first 
describe a slightly larger class of almost-K\"ahler $4$-manifolds.

Let $(\Sigma,g_{\Sigma},I_{\Sigma})$ be a Riemann surface with area form $\o_{\Sigma}$; we equip $\bb{R}^2$ with co-ordinates $x,y$ and let $z=x+iy$. For any $w:\bb{R}^2 \times \Sigma \to 
 \{z \in \bb{C}: \vert z \vert <1\}$ we consider the almost K\"ahler structure $(g,I)$ on $\bb{R}^2 \times \Sigma$, where
\begin{equation} \label{ak4}
\begin{split}
g=&\frac{4}{1-\vert w \vert^2}(dz-w d \overline{z}) \odot (d\overline{z}-\overline{w}dz)+g_{\Sigma}\\
g(I \cdot, \cdot)
=&\frac{i}{2}dz \w d\overline{z}+\o_{\Sigma}.
\end{split}
\end{equation}
We assume $w$ holomorphic in the $\Sigma$-variable, that is $I_{\Sigma}d_{\Sigma}w=id_{\Sigma}w$ where $d_{\Sigma}$ is differentiation on $\Sigma$; then $(g,J)$ is K\"ahler, where 
$$g(J \cdot, \cdot)=-\frac{i}{2}dz \w d\overline{z}+\o_{\Sigma}.$$
These examples generalise the construction in \cite{Apostolov-AD:integrability} where $w$ is chosen to depend only on $\Sigma$; also, they particularise the twisting construction in \cite{Chiossi-Nagy} to the case of the trivial line bundle over a Riemann surface.

\begin{teo} \label{t-ak} Let $(M^4,g,I)$ be almost K\"ahler with  $\dim \widetilde{\mathfrak{hol}} \leq 1$.Then:
\begin{itemize}
\item[(i)] $I$ is integrable, or
\item[(ii)] $g$ is Ricci-flat and selfdual, or
\item[(iii)] $(g,I)$ is locally given by \eqref{ak4}, where $w$ does not depend on $\bb{R}^2$ and the metric $(1-\vert w \vert^2)^{\frac{1}{4(\alpha-1)}}g_{\Sigma}, \alpha \in \bb{R} \backslash \{\pm 1\}$ is flat, or
\item[(iv)] $(g,I)$ is locally given by \eqref{ak4}, where $g_{\Sigma}$ is flat.
\end{itemize}
\end{teo}
\begin{proof}
By the previous results only the case $\widetilde{\mathfrak{hol}}=\bb{R}(\alpha \o_I+\o_J)$, where $\alpha$ is a real number and $J$ is an orthogonal, negative K\"ahler structure, has to be looked at. We shall also assume 
that $I$ is non-integrable.

The rank-two distributions $\D_{\pm}=\ker(IJ\mp \textsl{Id})$ are parallel for the canonical Hermitian connection and allow to decompose 
$\o_I=\o_{+}+\o_{-}, \ \o_J=-\o_{+}+\o_{-}$. 

Since $(g,I)$ is almost K\"ahler, $\eta_{IX}IY=-\eta_XY$ for all $X,Y$ in $TM$; it follows that the restrictions of $\eta$ to $\D_{\pm}$ are symmetric. Since the latter are $\widetilde{\nabla}$-parallel 
the distributions $\D_{+}$ and $\D_{-}$ must be both integrable. In particular the Levi-Civita connections of the induced metrics coincide with the 
restrictions of $\widetilde{\nabla}$ to $\D_{\pm}$. Let $\s_{\pm}=2\langle \tiR(\o_{\pm}), \o_{\pm}\rangle $ be the corresponding scalar curvatures; from 
the general formula $\tiR=\frac{1}{2}\rho^J \otimes \omega_J+\frac{1}{2} \tilde{\gamma_1} \otimes \omega_{I}$ we get 
\begin{equation} \label{s+}
-\s_{+}=\langle \tilde{\gamma}_1-\rho^J, \o_{+}\rangle, \ \s_{-}=\langle \tilde{\gamma}_1+\rho^J, \o_{-}\rangle.
\end{equation}
\indent
We now parametrise $\rho^J=\frac{\s}{4}\o_J+\mu \o_I+\varphi_1$, with $\varphi_1$ in $\lambda^2_I$; in particular $\rho^I=\frac{\s}{4}\o_I+\mu \o_J$. Let us also write 
$W^{+}\o_I=\frac{\kappa}{6}\o_I+\varphi_2,$ where $\varphi_2$ belongs to $\lambda_I^2$ and $\kappa$ is the conformal scalar curvature of $(g,I)$. We fix a local gauge $I_1$ for $I$ with connection forms $a$ and $b$ and 
impose $\tilde{\gamma}_1=\alpha \rho^J$. Since $(g,I)$ is almost K\"ahler, 
using \eqref{28} with $c=Ia$  and $\Phi=a \w Ia$ we get
\begin{equation} \label{rels}
\begin{split}
&\varphi_2=\alpha \varphi_1\\
& a \w Ia=(\alpha \mu-\frac{\s}{12}-\frac{\kappa}{6})\o_I+(\frac{\alpha \s}{4}-\mu)\o_J.
\end{split}
\end{equation}
Note that $\kappa=\s+6 \vert a \vert^2$. 

In particular $Ia$ is orthogonal to $J_1a, J_2a$ where $J_1$ is a local gauge for $J$ and $J_2=J_1J$. Since it is also orthogonal to $a$ it follows that 
$Ia=\pm Ja$ on the open set where $a \neq 0$.
\begin{enumerate}
\item[a)] $Ia=-Ja$. Then $a$ belongs to $\Lambda^1\D_{+}$, therefore $\D_{-}$ is the K\"ahler nullity of $(g,I)$, that is $\eta_{\D_{-}}=0$. Since $\D_{-}=\ker a \cap \ker(Ia)$
is integrable, the structure equations 
$$ da+b \w Ia=R(\o_{I_2}), \ d(Ia)-b \w a=R(\o_{I_1})
$$ 
of $(g,I)$ imply that $R(\o_{-})$ is orthogonal to $\lambda_I^2$. But the component in $\lambda_I^2$ of $R(\o_{-})=\frac{1}{2}(R(\o_{I})
+R(\o_{J}))=\frac{1}{2}\{h,I\}+\frac{1}{2}(W^{+}(\o_I)+\rho^J)$, is precisely $\frac{1}{2}(\varphi_1+\varphi_2)$ which therefore must vanish.
\begin{enumerate}
\item[a$_1$)] When $\alpha \neq -1$, we have $\varphi_1=\varphi_2=0$ by \eqref{rels}. This means $\Ric$ is $I$-invariant and $\o_I$ is an eigenform of $W^{+}$.  These almost K\"ahler manifolds form the so-called class 
$\mathcal{AK}_{3}$; using  their classification in \cite{Apostolov-AD:integrability} we get that $(g,I)$ is locally given by  \eqref{ak4}, where 
$w$ depends only on $\Sigma$. Consequently $\Ric=0$ on $\D_{+}$, and as 
$a \w Ia=-\vert a \vert^2 \o_{+}$, the scalar relations in \eqref{rels} are equivalent to $\frac{\alpha-1}{2}\s=\vert a \vert^2$. But 
$$ \vert a \vert^2=\frac{\vert d \vert w \vert \vert^2}{2(1-\vert w \vert^2)^2},
$$
and since $w$ is holomorphic it is easy to check that $(dI_{\Sigma}d) \ln (1 -\vert w \vert^2)=8\vert a \vert^2 \o_{\Sigma}$. The flatness of $(1-\vert w \vert^2)^{\frac{1}{4(\alpha-1)}}g_{\Sigma}$ follows now from 
the conformal transformation rule of the scalar curvature. 
\item[a$_2$)] When $\alpha=-1$ the bundle $\D_{-}$ is flat for the canonical connection. The second equation in \eqref{rels}, now equivalent to 
$2\mu+\frac{\s}{2}+\vert a \vert^2=0$, contains no further information; its left hand side computes in fact $\s_{-}$ by \eqref{s+}.

Pick locally-defined unit vectors $e_2=Ie_1$ such that $\widetilde{\nabla}
e_k=0, k=1,2$. Because $(g,I)$ is almost K\"ahler with K\"ahler nullity $\D_{-}$, it follows that the dual forms satisfy $de^1=de^2=0$. Now write locally 
$M=\Sigma \times \bb{R}^2 $ for some $2$-dimensional manifold $\Sigma$ where the co-ordinates $x,y$ on $\bb{R}^2$ are such that $e_1=\frac{d}{dx}, \ e_2=\frac{d}{dy}$.
Since the distribution $\D_{+}$ is also integrable $\Sigma$ can be chosen to correspond to $\D_{+}$, that is to admit local co-ordinates $t,u$ such that $\D_{+}=\span \{\frac{d}{dt}, \frac{d}{du}\}$ and  $\o_{+}=dt \w du$.
Then $J(\frac{d}{dx})=\frac{d}{dy}$ whereas on $\Sigma$ the complex structure 
$J$ is determined by a family of complex structure compatible with $\o_{+}$, possibly depending on $x,y$. Therefore $J=(1-S)^{-1}I_0(1-S)$ on $\Sigma$ where $I_0$ is given by $I_0(du)=dt, S=\biggr ( \begin{array}{lr} \Re \ w & \Im \ w \\ 
\Im \ w & -\Re \ w  \end{array} \biggr )$ in the basis $\{du, dt\}$, and $w:\bb{R}^2 \times \bb{R}^2 \to \{z: \vert z \vert<1\}$. Now requiring $J$ to be integrable amounts to 
$$ J_y=J J_x $$
on $\Sigma$. However this linearises as 
$S_y=I_0 S_x$ and the claim follows easily.
\end{enumerate}
\item[b)] $Ia=Ja$. Replace $J$ by $-J$ and apply part (a). Note that $\alpha$ transforms into $-\alpha$. \qedhere
\end{enumerate}
\end{proof}
Part (ii) in the theorem above is a manifestation of a closed '$4$-frame', whose local geometry is more complicated.  
The only known explicit examples are given by the Gibbons-Hawking Ansatz for a translation-invariant harmonic map, see \cite{Armstrong:ansatz} and its generalisations 
\cite{WSD}.

In the compact case theorem \ref{t-ak} can be enhanced as follows.
\begin{teo} \label{t-akcompact}
A compact almost K\"ahler $4$-manifold $(M,g,I)$ with $\dim \widetilde{\mathfrak{hol}} \leq 1$ must be K\"ahler.
\end{teo}
\begin{proof}
If $\widetilde{\mathfrak{hol}}=\{ 0 \}$ this is granted by corollary \ref{cor1}. Assume now that $\widetilde{\mathfrak{hol}}$ is generated by $F=F_0+\alpha \o_I$. If $F_0=0$ 
the metric $g$ is Ricci flat by proposition \ref{hol-vhk} and the integrability of $I$ follows from \cite{Sekigawa:aK}.

There remains to treat the case when $F_0 \neq 0$, when after re-normalisation(see section 5.3) we may assume that $F_0=\o_J$ where $(g,J)$ is K\"ahler compatible with the negative orientation. If $\alpha \neq \pm 1$ a case by case inspection of the 
proof of theorem \ref{t-ak} shows that $\varphi_1=\varphi_2=0$ in $\{x \in M:\eta_x \neq 0\}$. Now in any open set $U$ where $\eta=0$ the structure $(g,I)$ is K\"ahler 
hence $g$ is a local product; from the definition of $\varphi_1$ and $\varphi_2$ it is easy to see they vanish in $U$ as well. By a standard density argument we conclude 
that $\varphi_1=\varphi_2=0$ over $M$ hence the almost K\"ahler structure $(g,I)$ belongs to the class $\mathcal{AK}_3$. Because $M$ is compact it follows that $(g,I)$ is K\"ahler  by the classification results of Apostolov-Armstrong-Dr\u aghici in \cite{Apostolov-AD:integrability}. To complete the proof there remains to examine the following cases.
\begin{enumerate}
\item[a)] $\alpha=-1$. We will first show that $\eta_{\D_{-}}=0$ everywhere in $M$; note that it suffices to prove this at points where $\eta \neq 0$. Working around such points and using the same local choices as in the proof of theorem \ref{t-ak}  the second equation in \eqref{rels} reads 
\begin{equation*}
a \w Ia=-\frac{\kappa-\s}{6}\o_{+}-(2\mu+\frac{2\s+\kappa}{6})\o_{-}=-\vert a \vert^2 \o_{+}-(2\mu+\frac{2\s+\kappa}{6})\o_{-}.
\end{equation*}
In particular $\langle a \w Ia, \o_{+}\rangle+\vert a \vert^2=0$ after taking the scalar product with $\o_{+}$. If $Ia=Ja$ the form $a$ vanishes on $\D_{+}$ hence 
$\langle a \w Ia, \o_{+}\rangle=0$ and further $a=0$, contradicting the assumption that $\eta \neq 0$. Therefore around points where $\eta \neq 0$ we have 
$Ia=-Ja$ hence $\eta_{\D_{-}}=0$. The claim on the vanishing of $\eta_{\D_{-}}$ in $M$ is therefore proved. Since $\D_{-}$ is parallel w.r.t. to $\widetilde{\nabla}$ it follows that $\D_{-}$ is totally geodesic. 

Now having $\alpha=-1$ means that bundle $\D_{-}$ is flat w.r.t. the canonical connection, in particular $\s_{-}=0$. In the terminology of \cite{MZ} the foliation induced by the integrable distribution $\D_{+}$ is transversally totally geodesic with vanishing transverse Ricci curvature. Because the K\"ahler manifold $(M,g,J)$ is compact 
proposition 2.1 from \cite{MZ}, applied to the foliation induced by $\D_{+}$, shows that the latter is parallel w.r.t the Levi-Civita connection of $g$. Equivalently, $(g,I)$ is K\"ahler and the theorem is proved in this case.

\item[b)] $\alpha=1$. Replace $J$ by $-J$ and apply part a) above.
\end{enumerate}

\end{proof}
\begin{acknowledgements}
The authors thank Simon Salamon for suggestions and the ensuing conversations. 
It is a pleasure to acknowledge the hospitality of Uwe Semmelmann at  
various stages of this work and that of  the Department of Mathematics at the University of Auckland. We are grateful to the referee for useful 
suggestions on how to improve this paper.
\end{acknowledgements}

\end{document}